\documentclass[12pt, leqno]{amsart}

\usepackage{amsfonts}
\usepackage{amssymb}
\usepackage{amsmath}
\usepackage{amscd}
\usepackage{float}
\usepackage{graphicx}
\usepackage{amsfonts}

\usepackage{color}
\usepackage{cancel}
\usepackage[cp1250]{inputenc}
\usepackage[T1]{fontenc}

\usepackage{multirow}
\usepackage{multicol}
\usepackage{hyperref}

\usepackage{color}
\usepackage{xypic}

\usepackage[cmtip,arrow]{xy}

\usepackage[normalem]{ulem}
\newtheorem{theorem}{Theorem}
\newtheorem{corollary}{Corollary}

\newtheorem{proposition}{Proposition}
\newtheorem{definition}{Definition}
\newtheorem{remark}{Remark}

\newtheorem{notation}{Notation}

\def\P{\mathcal P}
\def\T{\mathcal{T}}

\def\Q{\mathcal{Q}}

\def\M{\mathcal M}

\def\P{\mathcal{P}}

\begin{document}

\title{Fuss-Catalan  numbers  and  planar partitions}

\author{F. Aicardi}
 \address{Sistiana 56, Trieste IT}
 \email{francescaicardi22@gmail.com}

\begin{abstract} We show  how the  Fuss-Catalan numbers $ \frac{1}{p n+1}\binom{pn+1}{n}$ enter  different problems of  counting simple  and  multiple planar partitions.
\end{abstract}

\date{}
\maketitle

\section{Results}  The Catalan  numbers  occur in  numerous  counting problems  as well  as  their  generalizations, the Fuss-Catalan numbers with one parameter, $p$ $$A_n^p:= \frac{1}{p n+1}\binom{pn+1}{n}.$$
The  Catalan numbers  correspond to the  case $p=2$.

A {\it planar partition} (or {\it non crossing} partition) is  a  set partition  that can be  represented  by a diagram whose  arcs, defining the  blocks,  do not cross each other, see Section \ref{pla}. 

In the context of  planar partitions, the numbers $A^p_n$  turn out  in  different  situations,  some of which we highlight here.     The  set of integers $\{1,\dots, n\}$ will be abbreviated as   $[n]$.

\noindent {\bf R1.}  The  Catalan  number    counts  the  number  of  planar partitions  of $[n]$.

\noindent {\bf R2.}\cite{Jon} The  Catalan  number counts the  number of  planar partitions of $[2n]$  whose  blocks  contain exactly  two  elements.

\noindent {\bf R3.}\cite{Cal} The  Fuss-Catalan number  $A^3_n $   counts the  number of  planar partitions of $[2n]$  in which each block  contains an  even number of  elements.

  We recall  that  a  double planar partition of $[n]$ is  a pair  of planar partitions  in which  the  first one is  a {\it refinement} of  the  second one,  i.e.,  each   block  of  the  first  partition  is  contained  ia  a  block of the  second  partition.

    In \cite{Aic1}, I have proved  that

\noindent {\bf R4.} the Fuss-Catalan number  $A^4_n $ counts  the  number of double planar partitions of [2n],   in  which   each  block of  the  first partition  has  two  elements, and   the  second partition is simply planar.

 This  was  done  by introducing  by  recurrence a  triangle of integers   whose elements of the $n$-th row  add  to   $A_n^4$.

 In \cite{Aic2} I  have  generalized  this  recurrence  to  every  integer  $p>0$ obtaining  all  Fuss-Catalan  numbers  $A_n^p$.  In the  same  note I  proved, using the recurrence  corresponding  to the  case $p=3$, that

\noindent {\bf R5.} $A^3_n$  counts  the  number  of  double  planar partitions of  $[n]$.

  So,  the question arises whether the mentioned  results  R1--R5    are  special cases  of  more general  statements.

This  note  answers  this  question,  proving in particular  four  statements, in which   the following  notations are  used.
\begin{enumerate}
 \item     A  {\it p-partition}  is  a  partition of  $[pn]$  whose  blocks  have  cardinality  $p$.

 \item  A {\it m-tuple planar  partition} of $[n]$ is an ordered set of $m$  planar partitions of $[n]$ such  that each one of  the first $m-1$ partitions is a refinement  of the successive.
  \end{enumerate}

\begin{theorem}\label{T1} For every integer $p>0$, the Fuss-Catalan number $A_n^p$ counts  the number of planar   $p$-partitions  of  $[pn]$.
  \end{theorem}

This  generalizes result  R2, see  Section \ref{S1}.

\begin{theorem}\label{T2} For   $p=2q$,    $A_n^p$   counts   the double   planar partitions of  $[qn]$, in which  the first one is a q-partition  and  the second one  is  simply planar.
\end{theorem}
 So,  this  theorem  generalizes  R4 ($q=2$) and  likewise it has a generalization:

 \begin{theorem}{\label{T3}} For  every pair $(m,p)$  of positive integers,     $A_n^{mp}$   counts   the $m$-tuple  planar partitions of  $[pn]$, in which  the first one is  a  $p$-partition and  the other ones are    simply planar.
\end{theorem}

  Theorem \ref{T3} has  the  following
\begin{corollary}{\label{C1}} For  every integer $m>1$    $A_n^{m}$   counts   the $(m-1)$-tuple  planar partitions of  $[n]$.
\end{corollary}

This  generalize  result  R1 ($m=2$)  and  R5  ($m=3$), see  Section \ref{S3}.

Let  $\M^p_n$ be  the  set  of  planar  partitions  of  $[pn]$, in  which  each  block  contains  a  number of  elements  multiple of  $p$.

\begin{theorem}{\label{T4}} For  every integer $p>0$, the  cardinality of   $\M^p_n$   is  $A_n^{p+1}$.
\end{theorem}
This  generalizes     result  R3 ($p=2$).

In   Figure  \ref{F0}   we  illustrate the  results above by giving  different  examples of  the  occurrence of  $A_2^6=6$: (a)  by Theorem \ref{T1}, it is the  number  of $6$-planar partitions  of $[12]$;  (b) by Theorem \ref{T2},  the  number of  double partitions  of $[6]$  the first of  which is a planar $3$-partition; (c)  by Theorem \ref{T3},  the  number of triple partitions  of $[4]$  the first of  which is a planar $2$-partition; (d)
by Corollary \ref{C1}, the  number  of 5-tuple planar partitions of  $[2]$; (e) by Theorem \ref{T4},   the number of planar partitions of [10] whose blocks  have  a number of  elements multiple of  $5$.

\begin{figure}[H]
 \centering
 \includegraphics[scale=0.9]  {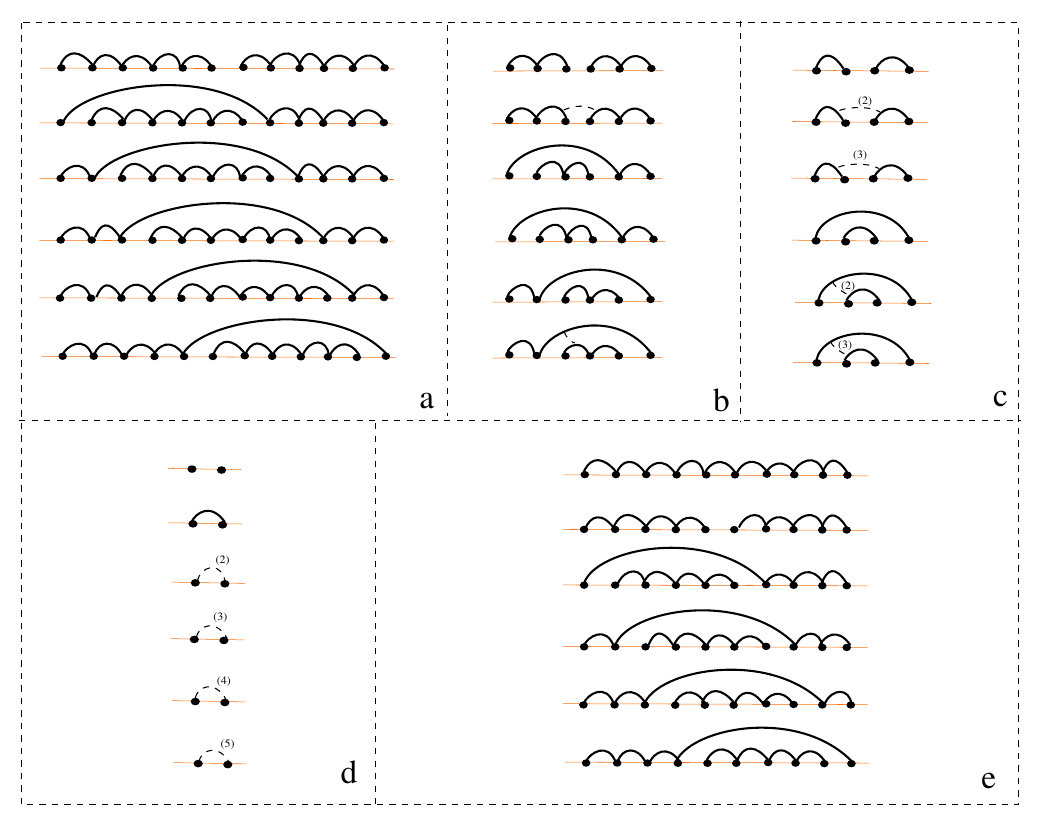}
 \caption{Five  examples of  the occurrence of  $A^6_2=6$.}\label{F0}
\end{figure}

In Appendix we  show  also  a  bijection    between the set of  the planar $p$-partitions of $[pn]$ and the set of the  $p$-ary trees with $n$ internal  nodes.

\section{Counting planar   $p$-partitions: proof of  Theorem \ref{T1} }\label{S1}

\subsection{Planar partitions and  arc  diagrams}\label{pla}

A partition  of  $[n]$  is  represented here  by   a  diagram,  consisting of $n$ points on a line, labeled  1 to $n$,  and  arcs  connecting  pairs of  points.   If  two  points  are connected by an  arc,  then they  belong  to  the  same block of  the partition.   The  diagram  representing a partition is  evidently  non unique.  A diagram   is  said  {\it standard}  if  for  every  pair of  points $(i<j)$ belonging to  the  same  block,  there is  an arc $(i,j)$  if  and  only if in the  same  block  there is no some   $k$  such  that  $i<k<j$.    The  standard  diagram  is  unique.  In what  follows  we deal  only with standard diagrams, unless it be  specified,  see Figure  \ref{F1}b.

 \begin{notation}\rm  A  partition is  said  planar if its  diagram is planar, i.e.,  its  arcs do no cross  each other. The  set of planar partitions of $[n]$ is  denoted by  $\P_n$, and the set of planar  $p$-partitions of $[pn]$ is  denoted by  $\P^p_n$.
\end{notation}

\begin{figure}[H]
 \centering
 \includegraphics[scale=0.6]  {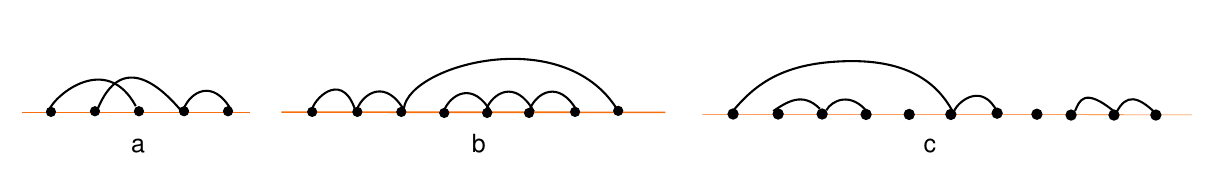}
 \caption{a)  Diagram of a non planar partition; b)   the  standard diagram of  a planar  4-partition; c) a partition of $[11]$  with 3 boxes  }\label{F1}
\end{figure}

Here we    prove Theorem \ref{T1},  generalizing  R1.

\begin{remark}\rm \label{p1} For $p=1$  a  planar  1-partition of $[n]$ consists of $n$  blocks, each one with 1 element.  Therefore it is unique and indeed  $A_n^1=1$ for  every $n$.
\end{remark}

\subsection{Proof of Theorem \ref{T1}} In \cite{Aic2} I have introduced, for  every  integer $p>0$, a   triangle $T^p(n,k)$,   whose rows  add to  $A^p_n$.

Here  we need to  consider,  for every  $p$,   the  triangle   $F^p(n,k)$,  obtained  by  $T^p(n,k)$ by:
\begin{equation}\label{reflex}    F^p(n,k):=  T^p(n,n-k), \quad  k=0,\dots,n.\end{equation}

The  integers $F^p(n,k)$  for   $n\ge 0$,  $0\le k \le n$,  result to  be given by the initial conditions
\begin{equation}\label{E0}
F^p(0,0)=1; \quad
F^p(n,0)=0 \quad \text{ for} \quad  n>0
  \end{equation}
  and  the  recurrence, for $n>0$  and  $0<  k \le n$:
 \begin{equation} \label{E2}   F^p(n,k)=  \sum_{j=k-1}^{n-1}  \binom{j-k+p-1}{p-2}F^p(n-1, j).
 \end{equation}
 
We recall  also  the  notion of  {\it box}  in a  diagram  of  a  planar partition.  Observe that   two  blocks $A,B$, satisfying  $min(A)<min(B)$  may  satisfy either  $max(A)<min(B)$      or  $max(A)>min(B)$.  In the last case  we have  also, by planarity,   that $max(A)>max(B)$, and  we  say that  $B$  is  nested  in $A$. A  box  is  a  block  which is  not  nested.
A  planar partition  of  $[n]$  with  $m$  blocks  has  $1\le k\le m$  boxes, see Figure  \ref{F1}.

We  denote  by  $N^p(n,k)$ the  number of  $p$-partitions  of  $[pn]$ having $k$  boxes.  We  see that  $N^p(n,k)$  fulfills  (\ref{E0}): indeed, the  void partition  has no boxes,  so $N^p(0,0)=1$; also,  there  are no $p$-partitions  of $[pn]$  with  no boxes,  so $N^p(n,0)=0$.

We see  now that  $N^p(n,k)$  satisfies  recurrence (\ref{E2}).
We consider the set of  partitions   of $\P^p_{n}$  having  $k$  boxes. We  observe that  each one  of them can be  obtained from a  unique partition   $Q\in \P^p_{n-1}$ having $j$  boxes  if $j\ge k-1$ and  $j\le n-1$,   by  adding  a  new  block $B$ to  $Q$  that has  the last  element  at  right  of  $Q$.
On the other hand,  given such a partition    $Q$ with   $j$  boxes,   we  may get different  partitions  with $k$  boxes.  Indeed, the  first  $k-1$ boxes  of $Q$  remain unaltered  in the  partition of $\P^p_{n}$,  say $P$,  while    $j-k+1$ boxes  must be  nested  in  the new  block $B$ of  $P$.   So, the  first point  of  $B$ is  necessarily  between the  first $ k-1$ boxes  and  the nested  $j-k+1$ boxes of $Q$.  Now,  between  the  first  and the last point of the new  block, there  are   $j-k+1$  boxes  and $p-2$  points of $B$, see Figure \ref{F2}.  The  partition $P$  is  uniquely defined by   the  ordered set containing  the $p-2$ points  and  the  $j-k+1$ boxes: for a fixed $Q$,  the  number  of possibilities  is  given  by
$\binom{j-k+1+p-2}{p-2}$.  This  gives the  recurrence (\ref{E2}).  Therefore  $N^p(n,k)=F^p(n,k)$.  Since  $|\P^p_n|$ is the sum of  the  numbers $N^p(n,k)$  for  $k=1, \dots, n$,  this  sum is $A^p_n$ by   (\ref{reflex})  and  (\ref{E0}).

\begin{figure}[H]
 \centering
 \includegraphics[scale=0.5]  {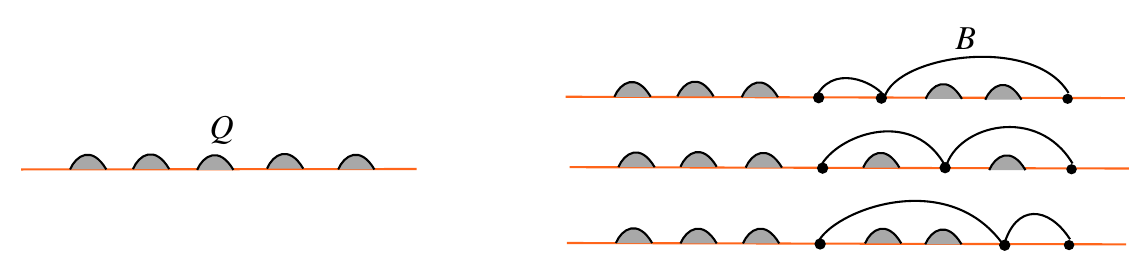}
 \caption{From a  partition  with  $j=5$  boxes and  $n$ blocks  to three  3-partitions  with  $k=4$  boxes  and  $n+1$ blocks. }\label{F2}
\end{figure}

\section{Counting  double planar  partitions}\label{S2}

\subsection{Double planar partitions  and  tied  arc diagrams}
In this  section  we  consider  special double planar partitions $(P_1,P_2)$ of  $[pn]$ in  which  $P_1\in \P^p_n$  and  $P_2\in \P_{pn}$. The  set of  such double partitions is  denoted by  $ \P^{II,p}_n$.

A  double partition  is  represented  by a  sole  diagram of  arcs  and ties  in the  following  way, see  also  \cite{AAJ}.  The partition  $P_1$  is  represented  by $n$ blocks.   Remember   that  each  block of $P_1$  consists of $p-1$ consecutive arcs. A block  is  labeled by  its point at left, so that  the blocks  of $P_1$ are  ordered  by this  labeling.   A tie is  an arc, drawn as a dotted line,     with  endpoints on two blocks of  $P_1$.      Suppose that  $B_{k_1},B_{k_2},\dots,B_{k_m}$  are   the ordered blocks of  $P_1$  contained in a same block of  $P_2$.  Then  we put a tie  from $B_{k_i}$  to  $B_{k_{i+1}}$,  for  $i=1,\dots, m-1$.   The planarity of  $P_2$  guarantees that the  ties can be drawn without crossing  neither the  arcs nor each other.  Moreover, a  double partition is uniquely represented  by  arcs and  ties in this  way, see  Figure \ref{F3}  for an example.

Now we  will prove that, for   $p>0$, the  cardinality of  the set $ \P^{II,p}_n$   is the Fuss Catalan number  $A^{2p}_n$.

\subsection{Proof  of  Theorem \ref{T2}}

The  proof given in  \cite{Aic1} for  the  case $p=4$  uses  the corresponding  Fuss-Catalan  triangle  and  here  we  could proceed the  same  way  for any  value of  $p$.   However,  we  prefer  to  proceed  by using  Theorem \ref{T1} an by proving that there is    a bijection    between  $\P^{2q}_n$  and $\P^{II,q}_n$.

We  need  to introduce  some basic notations.
The notion of refinement provides  the  set of  partitions of  $[n]$  with  a  partial  order, that  can  be  restrict to $\P_n$: $P_1\le P_2$  if  $P_1$ is  a  refinement of  $P_2$.  Also,   we  call  {\it product}   of two  planar partitions  $P_1$  and  $P_2$ the planar  partition  $P_3$ which  is  the minimal  partition  in $\P_n$ satisfying  $P_1 \le P_3$ and  $P_2 \le P_3$.  The product,  indicated  by the  symbol $*$,  is  evidently commutative.
 
\begin{remark}\label{I} Observe  that  the  partition $I  \in \P_n$ consisting  of  $n$ blocks  with  a unique  element,  satisfies  $I*P=P$  and   $I\le P$ for  every  $P\in \P_n$.  Moreover,  for  any  set of  partition  $\Q$,  the  set  of  double partitions $(I, Q)$  with  $Q\in \Q$  is  in bijection with  $\Q$.
\end{remark}

Finally,   given $P\in \P_{2n}$,  we  denote  by $P^{/2}$   a  partition   of $[n]$ obtained by  $P$ eliminating  from    all blocks  of  $P$ the even numbers and   sending   the  odd numbers $2j-1$  to $j$.

Consider  a partition $P\in \P_n^{2q}$.  The  elements of  every block  are  ordered  by  the  natural order in $[2n]$. Let us define the partition  $P'$ of $[qn]$  as  $P'=P^{/2}$.    

\begin{proposition}  The   partition $P'$   is  a  $q$-partition.
\end{proposition}
\begin{proof} The planarity of $P'$ follows  from the planarity  of $P$.  The  fact  that  $P'$ has exactly   $n$  blocks  with $q$  elements follows  from the  fact  that the   parity  of  the elements  $e_1, \dots, e_{2q}$ of  every  block of  $P$ alternates, since,  by planarity,  $e_{i+1}= e_i +  1 +  2kq $, for  some $k\ge 0$.  Therefore  each  block  of  the  partition $P'$  has  exactly  $q$  elements.
\end{proof}
 We  have  to  define  the  second partition, $P''\in \P_{qn}$,  of  the  double partition.

 Denote $I_2$ the partition  of  $[2nq]$  with  $qn$   blocks  of  2 elements  $\{ 2j-1,2j\}_{j\in [n] }$.  Then    consider  the  partition $\overline P= I_2 *  P$.  Clearly $P\le \overline P$.  The partition $P''$  of  $[qn]$ is thus  defined  as    $P'':=(\overline P)^{/2}$
 Diagrammatically,  we  will  put  a tie  between  two blocks  $B'_i$  and  $B'_j$  of $P'$  if    $  B_i$  and $  B_j$  result in the  same  block of  $\overline P$.

 This  procedure defines  uniquely the  partition $P''$  and  hence  the  double partition.
Example. In  Figure \ref{F3} we  show  how  from the  diagram of  the  partition  $P\in {\P^4_3}$, $ P=\{\{ 1,2,7,12\},\{3,4,5,6\},\{8,9,10,11\}  \}$, we get  the  diagram of   the    double partition    $(P',P'')\in  \P^{II,2}_3 $,   where $P'=\{\{1,4\},\{2,3\},\{5,6\}\}$, and  $P''=\{\{1, 4,5,6\},\{2,3\}\}$.

\begin{figure}[H]
 \centering
 \includegraphics[scale=0.6]  {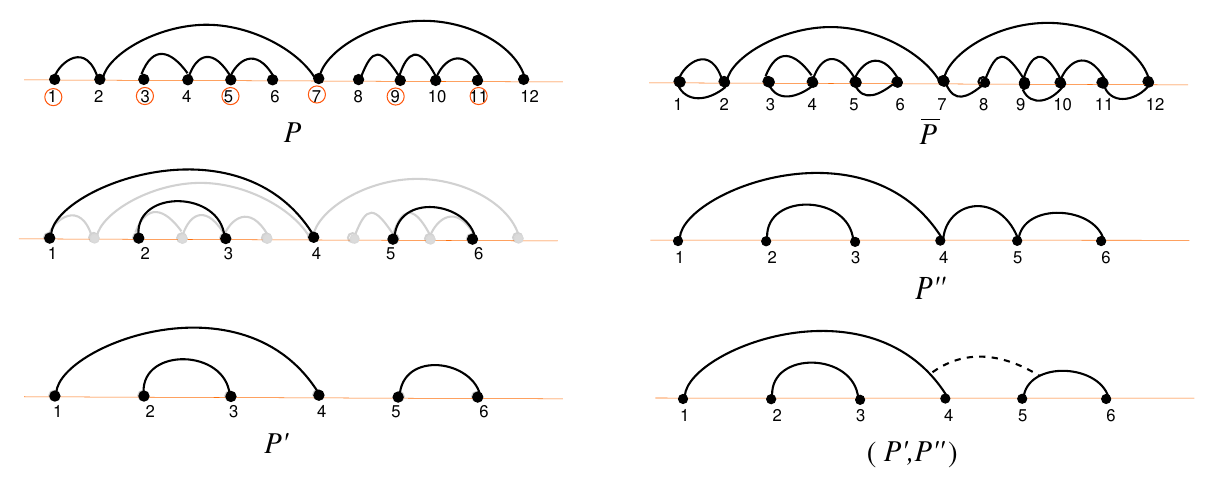}
 \caption{From the  diagram of  $P$  to  the  diagram of  $(P',P'')$. The  arcs of $I_2$  are  drawn below the  horizontal line. The  diagram of $\overline P$ is non standard   }\label{F3}
\end{figure}
The  next  observations   follow  from the  procedure  above introduced, as well as   from the  remark given proving Proposition \ref{p1}.

\begin{remark}\label{rem1}
\begin{enumerate}
\item
A  block $B_i'$ of  $P'$  is  non tied  to  other preceding blocks if and only if  the  first point of  $B_i$ is  odd;
\item  If  a  block $B'_j$ is  tied  to  a block  $B'_i$,  with  $j>i$,  then a)
the  first point of $B_j$ is  even and is the  next point of  an odd point  $y$ of $B_i$;
   the last  point of $B_j$ is odd  and  its next point is either  the even point of  $B_i$  successive to  $y$, or  the first  point of  another block  tied  to $B_j$.
\item The parity of  the  last  element of a block is  opposite to that of the  first one.
\end{enumerate}
  \end{remark}

To prove  that   the  procedure  above  defines  a bijection,  we  need to  verify that  the  inverse  procedure  defines  a  unique  planar $2q$-partition  $P$ starting  from  a  double planar partition  in which the  first partition $P'$ is  a $q$-partition and  the  second  one, $P''$, is planar.
Let us  start  from the  partition  $P'$  of  a  double partition in $\P^{II,q}_n$.  We  enumerate $B'_1, \dots, B'_n$ its  blocks   according  to their first points $x_1,\dots,x_n$,  and $B_1,\dots,B_n$ the  corresponding blocks  of $P$.  Remember that  the  block  $B_i$  surely contains  the odd points $2j-1$ for  all $j\in B'_i$.

To  define $P$,  we  have  therefore  to  put  the  even points  in each  block  $B_i$. We say that Remark \ref{rem1}  is sufficient to do  this.  Indeed,  if two consecutive  points $j,j+1$ belong to a  block $B'$,  the  even point $2j$ is  evidently  in $B$.  The problem arises   when we have to  put  in a  block of  $P$  a  point $2j$  such that $2j-1$ and  $2j+1$ belong to  different blocks $B'_r$  and  $B'_s$.  In this  case, we use Remark \ref{rem1}.  For  instance,  let us  start  from the double partition $(P',P'')$ in Figure \ref{F3}.  From   $P'=\{\{1,4\},\{2,3\},\{5,6\}\}$, we get the odd points of  the  three  blocks of $P$:      $B_1:\{1,7\},  B_2:\{3,5\}, B_3: \{9,11\}\}$.  Now, points 4  and  10  are  put in $B_2$ and  $B_3$. Since $B_1$ and  $B_3$  are tied, points 8 and  12 are put  in $B_2$ and  $B_1$  by  using item (2),  the  remaining point 2  belongs necessarily to $B_1$.

\section{Counting multiple planar  partitions}\label{S3}
\subsection{Proof  of  Theorem \ref{T3}}

The proof  of  Theorem \ref{T3} uses the  same argument  as   the proof of Theorem \ref{T2}:  we  define  a  bijection  between  $\P_n^{mp}$  and  $\P^{(m),p}_n$, the set  of  $m$-tuple partitions of $[pn]$,  of  which  the  first  one  is  a  $p$-partition,  and  the others  are  planar.

For  any  $k>2$, given $P \in  \P_{kn}$ we  denote  $P^{/k}$ the partition of  $[n]$  obtained by $P$ sending, for  every block,  each  element  of the form $kj+1$  to the element  $j+1$ of  a block of  $P^{/k}$. Observe that $0\le j \le n-1$.

Let $P\in \P_n^{mp}$, i.e.  $P$ is  a  planar partition of  $[mpn]$,  with $n$  blocks of  $mp$ elements.

Now,  we  define  the partitions   $I_r$, $r=1,\dots,m$  as  the partitions  of  $[mpn]$  containing  the  blocks
$\{ mk+1,mk+2,\dots,mk+r\}$  for  $k=0,\dots,pn-1$,  and coinciding  with  $I$ for  the  remaining elements.
Observe that  $I_1=I$.

Then  we  define, for  every $r=1,\dots,m$,  the  partition  of $[pn]$: $$P^{(r)}=  (P *  I_r)^{/m}$$
Observe that $I_1=I$, and   that  $P^{(1)}$ is  a  $p$-partition.  This last fact  follows  again from the  planarity of  $P$: the  number  of  elements  of   each  block  of  $P^{(1)}$  is  $p$  since  there  are  $p$  integers    in a  sequence  of    $mp$ integers  $x_i$  satisfying  $x_i=1 \mod m$   if  these  integer satisfy  $x_{i+1}= x_i+1 +  k  mp$  for  $k\ge 0$.

Since  $(P *  I_r)\le(P*I_{r+1})$,  the set $\{P^{(r)}\}_{r=1,\dots, m}$ is a set  of planar partitions ordered by  refinement,  in which $P^{(1)}$ is  a $p$-partition.

The proof   that  this  procedure defines  a  bijection  between  $\P^{mp}_{n}$  and  $\P^{(m),p}_n$  is  based  essentially on the   fact that  the  number of partitions     $\P_i\in \P^{mp}_{n}$ that  coincide  except  on  two blocks,  which are  sent by  $P^{/m}$ to the  same  two  blocks of a  partition   $P^{(1)}\in \P^{p}_n$ is  $m$  if and only if  these  blocks  of  $Q$  can  be  tied.  So, the $m$ partitions $P_i$  give rise  to the   different partition $P_i^{(r)}$ in which  they  become  tied.
The  proof  verifies that  the  products  of  $P_i$  with  $I_r$ cover all  these  possibilities. 

 In Figure \ref{F4}  we show  an example  with $m=3$, $p=2$  and $n=3$. Starting  from  a  $6$-partition  of  $[18]$,  we get  a   triple partition of  [6], of  which the  first  one is  a  2-partition.

\begin{remark}\rm  An  $m$-tuple partition, instead  of   a  set of $m$ partitions  ordered  by  refinement,
is  represented  by  a unique  diagram  with  arcs  and   ties, like the  double partition: the first partition  is represented as usually by an arc  diagram, and  every  tie   between  two  arcs is  labeled   by  the ordinal of the partition where  this  arc  appears.
\end{remark}

\begin{figure}[H]
 \centering
 \includegraphics[scale=0.6]  {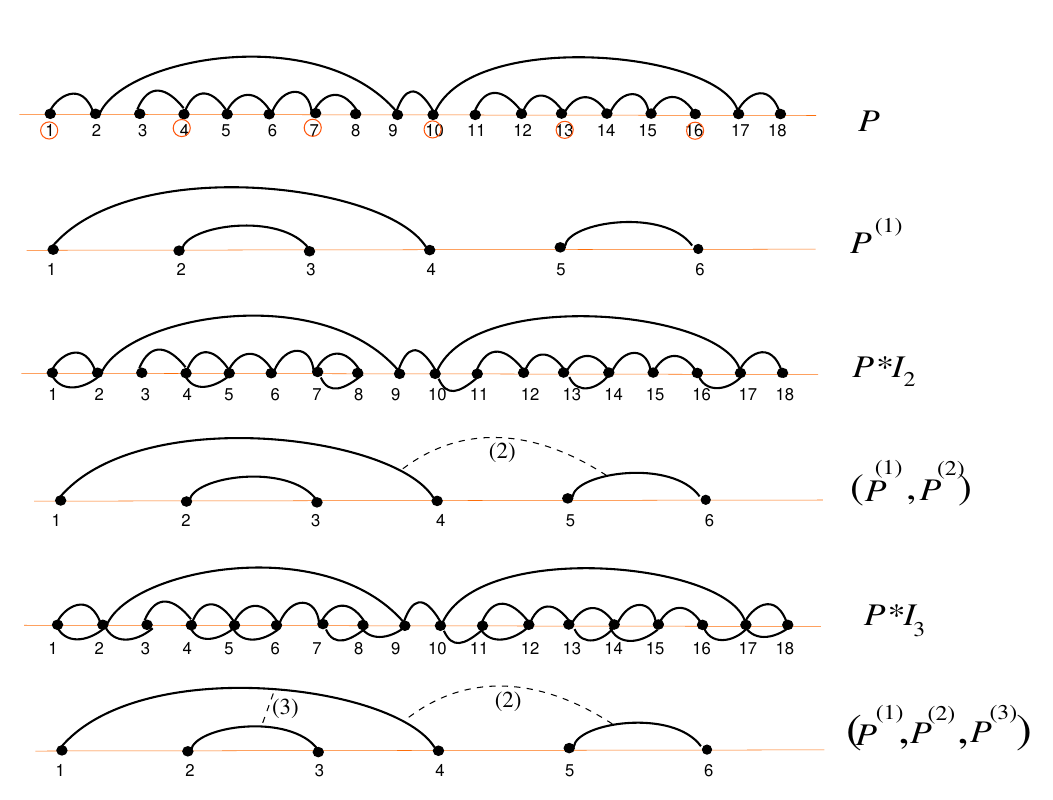}
 \caption{From the  diagram of  a 6-partition   to  the  diagram of  a triple partition. The  arcs of $I_2$  and $I_3$ are  drawn below the  horizontal line. The  diagrams of $P*I_j$ are non standard   }\label{F4}
\end{figure}

 \begin{proof}[Proof  of  Corollary \ref{C1}] When $p=1$,  Theorem  \ref{T3}  says  that $A^m_n$  counts  the  $m$-tuple  planar partitions of $[n]$  in which  the  first  partition is  the 1-partition,  therefore they  are  equivalent  to the  $(m-1)$-th planar partitions of $[n]$, see  Remarks  \ref{p1}  and  \ref{I}.
\end{proof}

\subsection{Proof  of  Theorem \ref{T4}}\label{S4}
 
 We  define  a  bijection  between  $\M^p_n$  and  the set of $p$-tuple planar partitions  of $[n]$. 
  Consider a planar partition $M\in \M^p_n$.    As  in the proof  of  Theorem  \ref{T3},  we  define the   partitions
  $$  M^{(r)}:=  (M *  I_r)^{/p}, \quad  r=1, \dots, p.  $$ 
 They  are  evidently  $p$  planar partitions of $[n]$ ordered by   refinement.   The  fact  that $M^{(1)}$  is a  planar partition of  $[n]$  follows  again from the planarity of  $M$  and   from  the  fact  that  in each block of $M$  with  $kp$ elements,  $k$ of them are  congruent to 1  $\mod p$,  see  proof of  Theorem  \ref{T3}.    
 
 Therefore  the  cardinality  of  $\M^p_n$ is  $A^{p+1}_n$  by  Corollary \ref{C1}.

\section{Appendix.  Planar p-partitions  and full   p-ary trees}

\begin{definition} \rm We call  for short {\it  $p$-tree} an ordered  tree where   every node has 0 or $p$ children (such trees  are  often called  full $p$-ary trees). The  set of $p$-trees  with $n$ internal  nodes (i.e.,nodes with $p$ children) is denoted $\T^p_n$.
\end{definition}

 In \cite{GKP}, it is shown that  the cardinality of $\T^p_n$ is $A^p_n$.  So,   another proof  of  Theorem \ref{T1} is  the  following

\begin{proposition} There  is  a  bijection  between the  set $\T^p_n$ of  $p$-trees  with  $n$ internal  nodes  and  $\P^p_n$
\end{proposition}
\begin{proof} Consider  the map  $\tau$ from  $\P^p_n$ to  $\T^p_n$ defined   in this  way.   Let  $D$  be  the  arc diagram  representing  a $p$-partition of [pn].  It consists of $pn$ points and   $n$ chains   of $p-1$ consecutive  arcs, that connect the  $p$ points of a same block.     For any two  adjacent points $x,x+1$ we say that   $x+1$  is the successive of $x$.  We  enumerate 1 to $n$  the blocks  according to the order of  their first points.   Then,  for  the $k$-th  block, we  denote by $(k,m)$, $m=1,\dots,p$,  its points   in the order from left to  right.   Now, we associate to $D$ a $p$-tree.  To  the  first  block we  associate  the root  of  the  $p$-tree  with $p$ children.  Observe now  that  the point (2,1) is  necessarily  at  right of  (1,1),  but is  the successive of (1,m), for some $1\le m \le p$.  Then we  label  by  2  the  $m$-th   child  of the root,  and we  give to it $p$ children. We  say  that  2  is the child of 1 of ordinal $m$.  Similarly,  point (3,1) is the successive either of  $(1,j)$ for $j>m$,  and in this  case  we label by  3  the  $j$-th child of  the root,   or of $(2,h)$  for  some  $1\le h \le p$, in which case  we label by 3  the $h$-th child  of node  2. When   we  give  a  label to  a  node,  we give to  it $p$ children. We proceed this  way  with all points $(k,1)$, till $k=p$, assigning the  label  $k$  to  the $i$-th  child of the node $k'<k$ such that $(k,1)$ is  the successive of  $(k',i)$.  Evidently  the label $k$  is  assigned to  a  node which is  internal  by construction. So, at  the end  we have  a $p$ tree  with $n$ internal  nodes.
 The map $\tau$ is  bijective  since  a $p$-tree  is  uniquely  defined  by  the $n-1$ pairs  $(a ,b )_k$, $k=2,\dots ,n$, where $a \in [n]$, $a<k$, is  the  father  of the node $k$  and  $b  \in [p]$  is  the ordinal of the  child $k$; on the other hand,  the diagram  of a  planar $p$-partition is  uniquely  determined by  the  same  pairs, meaning that  the point   $(k,1)$ is the successive of point $(a,b)$.
\end{proof}
Example.  See  Figure \ref{F5}  

\begin{figure}[H]
 \centering
 \includegraphics[scale=0.9]  {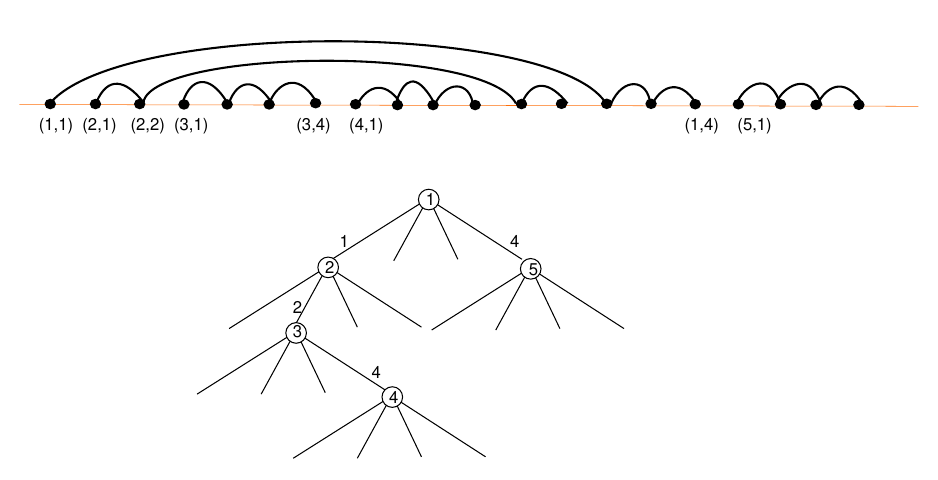}
 \caption{The diagram   of  a partition of $\P^4_5$  and  the  corresponding 4-tree  with  5 internal nodes   }\label{F5}
\end{figure}

\begin{remark}\rm  A  box   of  a  $p$-partition  $P$ correspond  evidently to a node of  $\tau(P)$  that  is   either the  root,   or a node  having    all  ancestors  with maximal   ordinal,    e.g. nodes 1 and  5  in the  figure  above.
\end{remark}


\begin{thebibliography}{25}


\bibitem{Aic1} F. Aicardi,  {\it Catalan triangles  and tied arc  diagrams},   arXiv:2011.14628,(2020), 16 pp.

\bibitem{Aic2} F. Aicardi,  {\it Fuss-Catalan triangles},   arXiv:2011.14628,(2023), 16 pp.

\bibitem{AAJ} F. Aicardi, D.Arcis, J. Juyumaya,
{\it Brauer and Jones tied monoids}, J. Pure Appl. Algebra 227 (2023), no. 1, Paper No. 107161, 25 pp.

 \bibitem{Cal} D. Callan,  OEIS A001764 (2007)

\bibitem{GKP}  R.L. Graham, D.E. Knuth, and O. Patashnik. {\it Concrete Mathematics}
Addison-Wesley, New York, 1989

\bibitem{Jon} V.F.R. Jones, {\it  Index for subfactors}, Invent. Math. 72 (1983), no. 1, 1–25.



\end{thebibliography}
  \end{document}